\documentclass[8pt]{article}
\usepackage{indentfirst,latexsym,bm}
\usepackage{amsfonts}
\usepackage{amssymb}
\usepackage{mathptmx} \DeclareMathAlphabet{\mathcal}{OMS}{cmsy}{m}{n} 
\usepackage[leqno]{amsmath}
\usepackage{dsfont}
\usepackage{amsthm}
\usepackage[all]{xy}
\usepackage{hyperref}
\usepackage{color,xcolor}
\newtheorem{theo}{Theorem}[section]
\newtheorem{prop}[theo]{Proposition}
\newtheorem{lemm}[theo]{Lemma}
\newtheorem{coro}[theo]{Corollary}
\theoremstyle{definition}
\newtheorem{defi}[theo]{Definition}
\newtheorem{exam}[theo]{Example}
\newtheorem{remk}[theo]{Remark}

\newtheorem{ques}[theo]{Question}

\newcommand{\A}{\mathcal{A}}

\newcommand{\C}{\mathcal{C}}
\newcommand{\B}{\mathcal{B}}
\newcommand{\D}{\mathcal{D}}
\newcommand{\E}{\mathcal{E}}
\newcommand{\F}{\mathcal{F}}
\newcommand{\cZ }{\mathcal{Z}}

\DeclareMathOperator{\FPdim}{FPdim}
\newcommand{\I}{\mathcal{I}}

\newcommand{\K}{\mathds{k}}

\newcommand{\Q}{\mathcal{O}}
\DeclareMathOperator{\Rep}{Rep}
\newcommand{\sVec}{\mathrm{sVec}}
\newcommand{\vvec}{\mathrm{Vec}}
\newcommand{\W}{\mathcal{W}}
\newcommand{\Y}{\mathcal{Z}}
\newcommand{\Z}{\mathbb{Z}}
\DeclareMathOperator{\BrPic}{BrPic}

\DeclareMathOperator\homology{H}
\renewcommand\H{\homology}

\begin{document}
\title{On   the structure of Witt groups and minimal extension conjecture}
\author{Theo Johnson-Freyd, Victor Ostrik, and Zhiqiang Yu}
\date{}
\maketitle

\abstract

Let $\E=\Rep(G)$ be a Tannakian fusion category. For a braided fusion category $\C$ over $\E$ we give sufficient and necessary conditions that characterize the Witt relation $[\C]=[\E]$. Then we show the   Witt group $\W(\E)$ is naturally a direct sum of Witt group $\W:=\W(\vvec)$ and the group $\H^4(G,\K^\times)$. Consequently, for any non-degenerate fusion category $\C$ over $\E$, there is a positive integer $n$  (e.g. $n=|G|$) such that $\C^{\boxtimes_\E^n}$ admits a minimal extension.
\bigskip


\section{Introduction}

Let $\K$ be an algebraically closed field of characteristic zero. In this paper we are concerned with the study of braided fusion categories over $\K$, see e.g. \cite{DrGNO2}. A complete classification of such categories is currently out of reach, so one can study such categories up to a suitable equivalence. One such equivalence relation (on a subset of {\em non-degenerate} braided fusion categories) is
{\em Witt equivalence} introduced and studied in \cite{DMNO}. The equivalence classes form the {\em Witt group}
$\W$.

Recall that a braided fusion category is non-degenerate if and only if its {\em M\"uger center} $\C'$ is trivial, i.e.\ $\C'=\vvec$. In general, $\C'$ is a {\em symmetric fusion subcategory} of $\C$. Thus for a given symmetric fusion category $\E$ one can study {\em non-degenerate braided fusion categories over} $\E$, i.e.\ braided fusion categories equipped with an equivalence (of symmetric fusion categories) $\E \simeq \C'$. In particular in \cite{DNO} the {\em relative Witt group} $\W(\E)$ was introduced. An important special case is
$\E =\sVec$; in this case the group $\W(\E)$ is called super Witt group and is denoted $s\W$. It was shown in \cite{DNO} that the group $s\W$ has a simpler structure than $\W$ and this is useful for the study of the group
$\W$ thanks to the canonical homomorphism $\W \to s\W$.

The main goal of this paper is to get a better understanding of the group $\W(\E)$ for a general~$\E$.
A well known theorem of Deligne \cite{Desym} implies
that symmetric braided fusion categories fall in two classes: {\em Tannakian} and {\em super Tannakian}. Assume that $\E$ is Tannakian, i.e.\ there exists a finite group $G$ and an equivalence of symmetric fusion categories $\E \simeq \Rep(G)$. Recall that the finite group $G$ is 
 determined by this assumption up to inner automorphisms; in particular the cohomology groups $\H^i(G,\K^\times)$ are canonically determined by the category $\E$, since the inner automorphisms of $G$ act trivially on cohomology. Here is our main result (see Theorem \ref{isomptheo}):

\begin{theo}\label{TannTh}
Let $\E \simeq \Rep(G)$.
There is a canonical isomorphism of abelian groups
$\W(\E)\cong\W \oplus \H^4(G,\K^\times)$.
\end{theo}

\begin{remk} This in particular proves Conjecture~4.3 of \cite{JR}, which briefly sketches  a proof subject to various assumptions about a not-yet-developed higher Morita theory. Also T.~D\'{e}coppet in
\cite[Proposition 5.3.9]{Deceoppt} used the theory of fusion 2-categories to construct a set bijection $\W(\E)=\W \times H^4(G,\K^\times)$ which is very likely to coincide with our isomorphism.
\end{remk}

The proof of Theorem \ref{TannTh} is as follows: first we define the canonical homomorphism $\phi_\E: \W(\E) \to \W$ and its canonical splitting, see Section \ref{Tannakian}. Then we use the theory of graded extensions
of fusion categories \cite{ENO2} in order to construct isomorphism $\operatorname{Ker}(\phi_\E)=\H^4(G,\K^\times)$.

\begin{remk} Theorem \ref{TannTh} gives an interpretation of the group $\H^4(G,\K^\times)$ in terms of the category $\E$. There are similar interpretations for $\H^i(G,\K^\times)$ for $i\le 3$, namely:

$\H^1(G,\K^\times)=$ group of isomorphism classes of invertible objects in $\E$ (since the invertible objects
of $\E$ are 1-dimensional representations of $G$ which are in bijection with $\mbox{Hom}(G,\K^\times)=H^1(G,\K^\times)$);

$\H^2(G,\K^\times)=$ group of equivalence classes of invertible module categories over $\E$, see \cite{Carnovale, DaN};

$\H^3(G,\K^\times)=$ group of equivalence classes of minimal modular extensions of $\E$, see \cite{LKW}.

Note that the interpretations above make sense for $\E$ which is not necessarily Tannakian and give
interesting invariants of super Tannakian categories $\E$ (it is not quite clear what should be the invariant
of super Tannakian category $\E$ corresponding to $\H^4$; provisionally this can be the kernel
of the natural homomorphism $\W(\E) \to s\W$, see Section \ref{sectsuperTann}).
It would be interesting to find such interpretations  for $\H^i(G,\K^\times)$ with $i\ge 5$. 
\end{remk}

Let $\C$ be a non-degenerate braided fusion category over $\E$.  In \cite[Conjecture 5.2]{Mu} M. M\"{u}ger conjectured that there exists a non-degenerate fusion category $\D$ such that $\C\subseteq\D$ and $\C_\D'=\E$, where $\C_\D'$ is the {\em centralizer} of $\C$ in $\D$. This is known as the {\em minimal extension
conjecture}. Drinfeld constructed counterexamples to this
conjecture for some Tannakian categories $\E$ in an unpublished work. On the other hand it is known
that the minimal extension conjecture is true in some cases: the case $\E=\Rep(G)$ where $\H^4(G,\K^\times)=0$
is a consequence of \cite{ENO2} and the case $\E=s\vvec$ is the main result of \cite{JR}.

In \cite{OY} and \cite{JR}, the minimal extension conjecture for symmetric fusion category $\E$ was connected with the surjectivity of certain group homomorphism between Witt groups $\W$ and $\W(\E)$. Namely, it was shown in \cite{OY} that the minimal extension conjecture is true for a braided fusion category $\C$ over $\E$ if and only if
the class $[\C]\in \W(\E)$ is in the image of natural homomorphism $S_\E:\W\to\W(\E)$, where $S_\E:[\D]\mapsto[\E\boxtimes\D]$ for any $[\D]\in\W$. In \cite{OY} this was used to show that the minimal extension conjecture holds for slightly degenerate weakly group-theoretical braided fusion categories. This result was
superseded by \cite{JR} where it was shown that the minimal extension conjecture holds for all slightly degenerate braided fusion categories. We refer the reader to \cite{GalVen} for further results on the minimal extension
conjecture in the super-Tannakian case.

As a consequence of Theorem \ref{TannTh} we get the following (see Corollaries \ref{h4minext} and \ref{existMinExten}):

\begin{coro}\label{minextconj}
\begin{enumerate}
\item[(i)] The minimal extension conjecture holds for all braided fusion categories
over Tannakian category $\E=\Rep(G)$ if and only if $\H^4(G,\K^\times)=0$.
\item[(ii)] For any braided fusion category $\C$ over Tannakian category $\E$ there exists a positive integer~$n$ such that $\C^{\boxtimes_\E^n}$ admits a minimal extension. In fact we can take $n=|G|$.
\end{enumerate}
\end{coro}

Finite groups for which $\H^4(G,\K^\times)=0$ include those for which all Sylow subgroups are cyclic, as well as the Mathieu group $\mathrm{M}_{23}$ \cite{Milgram}.


The paper is organized as follows. In Section \ref{Prelim} we recall some most used results about   braided fusion categories and Witt groups.
 In Section \ref{Tannakian} we give a characterization of the elements
of the Witt equivalence class $[\E]\in \W(\E)$. 
In Section \ref{Gqmc} we introduce the useful technical notion
of $G$-quasi-monoidal categories and in Section \ref{isom eta} we use it to prove our main results. Finally in Section
\ref{sectsuperTann} we discuss some aspects of the super Tannakian case.

{\em Acknowledgments:} Part of the work on this project was done while one of us (VO)\ visited the Yangzhou University; we are grateful for the hospitality. We are also grateful to Thibault D\'ecoppet, Vladimir Drinfeld, and Dmitri Nikshych for very useful conversations. ZY is supported by NSFC (no.\ 12571041) and Qinglan Project of Yangzhou University. TJF is supported by NSERC (no.\ RGPIN-2021-02424) and the Simons Collaboration on Global Categorical Symmetries (no.\ SFI-MPS-GCS-00008528-16).

\section{Preliminaries}\label{Prelim}

In this paper we will assume that the reader is familiar with basic notions of the theory of fusion categories
(such as notions of fusion category and braided fusion category), see e.g.\ \cite{EGNO}. 
Note that any fusion subcategory of a braided fusion category is automatically braided. 
We will in particular use M\"uger's theory of centralizers \cite{Mu}, \cite[8.20]{EGNO}. Thus for a fusion subcategory
$\D \subset \C$ of a braided fusion category $\C$ we will denote by $\D'_\C$ the full subcategory of $\C$
consisting of objects $Y\in \C$ satisfying $c_{X,Y}\circ c_{Y,X}=\mathrm{Id}_{Y\otimes X}$ for all $X\in \D$ (here
$c_{X,Y}: X\otimes Y\to Y\otimes X$ is the braiding in the category $\C$). We will also use $\D'$ for
$\D'_\C$ when the category $\C$ is clear from the context. The subcategory $\D'$ is again a fusion
subcategory of $\C$. In particular we have the subcategory $\C'=\C'_\C$ which is called {\em M\"uger center}
of $\C$; this is always a symmetric fusion category.
A {\em non-degenerate} fusion category is a braided fusion category $\C$ such that $\C'$ is trivial
(i.e.\ consists only of direct sums of the unit object). For instance the {\em Drinfeld center} $\cZ(\A)$ of a fusion
category $\A$ is non-degenerate. If $\C$ is non-degenerate we have the {\em double
centralizer property}: $(\D')'=\D$ for any fusion subcategory $\D \subset \C$.

Let $\E$ be a symmetric fusion category. Following \cite{DNO} we say that a braided fusion category $\C$ is
\emph{non-degenerate over $\E$} if it is equipped with an equivalence $\C'\cong \E$. An important example
is the relative Drinfeld center $\cZ(\A,\E)$ of a fusion category $\A$ over $\E$ (this means that
$\E$ is equivalent to a fusion subcategory of $\cZ(\A)$ which projects fully faithfully to $\A$ under
the forgetful functor $\cZ(\A)\to \A$, and  $\cZ(\A,\E):=\E'_{\cZ(\A)}$). For non-degenerate fusion categories
$\C$ and $\D$ over $\E$ there is an important notion of {\em Deligne tensor product over $\E$},
$\C\boxtimes_\E\D$ which is also a non-degenerate fusion category over $\E$, see \cite[2.5]{DNO}.
Namely, $\C\boxtimes_\E\D:=(\C\boxtimes\D)_A$ is the category of right $A$-modules, where
\begin{align*}
A=\oplus_{X\in\Q(\E)} X\boxtimes X^*\end{align*}
 is the canonical connected \'{e}tale algebra in $\E \boxtimes \E$ such that $(\E\boxtimes\E)_A\cong\E\boxtimes_\E\E\cong\E$, see \cite[Remark 2.8]{DNO}.

 Recall that two non-degenerate fusion categories $\C, \D$ over $\E$ are {\em Witt equivalent} if there exists
 an equivalence $\C \boxtimes_\E \cZ(\A_1, \E)\cong \D \boxtimes_\E \cZ(\A_2,\E)$ for some fusion
 categories $\A_1, \A_2$ over $\E$, see \cite[Definition 5.1]{DNO}. This is an equivalence relation
 and the equivalence class of the fusion category~$\C$ is denoted by~$[\C]$. The operation $[\C][\D]:=[\C \boxtimes_\E \D]$
 is well defined and it makes the set of equivalence classes $\W(\E)$ into an abelian group called the
 {\em relative Witt group}; this is the main character of this paper.

 Let $G$ be a finite group acting on a fusion category $\A$ (this means that there is a homomorphism
 of 2-groups $G\to \underline{\mbox{Aut}}(\A)$). Then one defines the {\em equivariantization} $\A^G$ to be the category of categorical fixed points. It
 is a fusion category over $\E=\Rep(G)$, see \cite[4.2.2]{DrGNO2}. Conversely, given a fusion
 category $\B$ over $\E =\Rep(G)$ there is a construction of {\em de-equivariantization}
 which gives a fusion category $\B_G$ equipped with a $G$-action, see \cite[subsection 4.2.4]{DrGNO2}.
 These two constructions are inverse to each other, see \cite[Theorem 4.18]{DrGNO2}.
 We will also use braided version of these constructions, see \cite[subsection 4.2]{DrGNO2}.


\section{Group homomorphism $\phi_\E:\W(\E)\to\W$}\label{Tannakian}

As observed already in \cite[Remark 5.8]{DNO}, the assignment $\E \mapsto \W(\E)$ is functorial: any functor $\E \to \F$ of symmetric fusion categories determines, via base-change, an abelian group homomorphism $\W(\E) \to \W(\F)$. The domain of this functor is most naturally the $2$-category of symmetric fusion categories, symmetric linear functors, and symmetric natural transformations (which are automatically natural isomorphisms because of the existence of duals); but since the codomain the the merely-$1$-category of abelian groups, one might as well treat the domain as the $1$-category of symmetric fusion categories and \emph{isomorphism classes} of symmetric linear functors between them.

For example, the unique symmetric functor $\vvec \to \E$ selects, for any $\E$, a canonical homomorphism $S_\E = \W(\vvec \to \E) : \W \to \W(\E)$. Suppose that $\E = \Rep(G)$ is Tannakian. A choice of fiber functor $\E \to \vvec$ selects a section-retract pair $\vvec \to \E \to \vvec$ (meaning that the composite $\vvec \to \vvec$ is the identity), and hence the induced map $\phi_\E = \W(\E \to \vvec) : \W(\E) \to \W$ together with $S_\phi$ are a section-retract pair of abelian groups.
Section-retract pairs of abelian groups are precisely the same as direct sum decompositions, and so:
 \begin{theo}[\cite{OY}]\label{Wittgroup}We have a split exact sequence of abelian  groups
\begin{align*}
1\to \text{Ker}(\phi_\E)\overset{i}{\to} \W(\E)\underset{S_\E}{\overset{\phi_\E}{\rightleftarrows} }\W\to 1.
\end{align*}
In particular,   $\W(\E)\cong\W\oplus \text{Ker}(\phi_\E)$. \qed
\end{theo}
\begin{remk}
  Explicitly, the homomorphisms $\phi_\E$ and $S_\E$ are given by \[ \phi_\E([\C])=[\C_G], \qquad S_\E([\D])=[\E\boxtimes\D],\] where $[\C]\in \W(\E)$ and $[\D]\in \W$ and $\E = \Rep(G)$.
\end{remk}

For arbitrary symmetric fusion $\E$, it was proved  in \cite[Theorem 3.1.3, Remark 3.1.4]{OY} that the  group homomorphism $S_\E:\W\to \W(\E)$ is surjective if and only if the  minimal extension conjecture is true for   non-degenerate fusion categories over $\E$. Moreover, if $\C$ admits a minimal extension $\D$, then $[\C]=[\E\boxtimes\D]$. However,  Drinfeld' counterexamples \cite{Drinfeld} show that $S_\E$ is not surjective for general Tannakian $\E$.

\begin{coro}\label{kernelminext} The map $S_\E$ is surjective (equivalently, the minimal extension
conjecture holds for all non-degenerate braided fusion categories over $\E$) if and only if
$\text{Ker}(\phi_\E)=0$.
\end{coro}

In general, $\text{Ker}(\phi_\E)$ is non-trivial \cite{Drinfeld}, hence it is crucial to know the structures of $\text{Ker}(\phi_\E)$. We begin with the following lemma, which will be used to characterize elements of $\text{Ker}(\phi_\E)$.

\begin{lemm}\label{G-equivalg}Let $G$ be a finite group, and let $\C$ be a braided fusion category
equipped with $G$-action. If $A\in\C$ is a $G$-equivariant  commutative algebra, then there is a well-defined action of $G$ on  $\C_A$ (as a monoidal category).
\end{lemm}
\begin{proof}It suffices to show $(g(X),m_{g(X)})\in\C_A$, where $m_X:X\otimes A\to X$ is the right $A$-module structure of $X$, and the morphism $m_{g(X)}$ is defined as follows:
\begin{align*}
m_{g(X)}:=g(m_X)\circ g(\text{id}_X\otimes \eta_{g^{-1}})\circ\mu_g(X,g^{-1}(A))\circ \text{id}_{g(X)}\otimes(\nu^A_{g,g^{-1}})^{-1}.
\end{align*}
Here and below we use $g$ to denote the tensor functor determined by the homomorphism $\rho:G\to \text{Aut}_\otimes^\text{br}(\cZ(\B))$; we write $\nu_{g,h}:g\circ h\overset{\sim}{\to} gh$ for the  isomorphism making $\rho$ into a morphism of 2-groups
 and $\mu_g(X,Y):g(X)\otimes g(Y)\overset{\sim}{\to} g(X\otimes Y)$ for the monoidality of the functor $g$. We furthermore denote by $\eta_g:g(A)\overset{\sim}{\to} A$ the $G$-equivariance of $A$. Finally, we will write $(A, m_A,u)$ be the algebra structure of $A$, that is,
\begin{align*}
m_A\circ\text{id}_A\otimes m_A=m_A\circ m_A\otimes \text{id}_A, m_A\circ\text{id}_A\otimes u=\text{id}_A =m_A\circ u\otimes \text{id}_A,
\end{align*}
and we 
%
%
omit the  associator and unit restriction of fusion category $\C$.

Thus we need to prove the following equations:
\begin{equation}
m_{g(X)}\circ\text{id}_X\otimes u=\text{id}_{g(X)}, \qquad m_{g(X)}\circ m_{g(X)}\otimes \text{id}_A=m_{g(X)}\circ m_A.
\end{equation}
For the first one of above equations, by definition of $m_{g(X)}$,
\begin{align*}
&m_{g(X)}\circ\text{id}_{g(X)}\otimes u\\
&=g(m_X)\circ g(\text{id}_X\otimes\eta_{g^{-1}}(A))\circ \mu(X,g^{-1}(A))\circ\text{id}_{g(X)}\otimes(\nu^A_{g,g^{-1}})^{-1}\circ\text{id}_{g(X)}\otimes u\\
&=g(m_X)\circ g(\text{id}_X\otimes\eta_{g^{-1}}(A))\circ \mu(X,g^{-1}(A))\circ\text{id}_{g(X)}\otimes g(g^{-1}(u))\circ\text{id}_{g(X)}\otimes (\nu^\mathbf{1}_{g,g^{-1}})^{-1}\\
&=g(m_X)\circ g(\text{id}_X\otimes\eta_{g^{-1}}(A))\circ g(\text{id}_X\otimes g^{-1}(u))\mu(X,g^{-1}(\mathbf{1}))\circ\text{id}_{g(X)}\otimes (\nu^\mathbf{1}_{g,g^{-1}})^{-1}\\
&=g(m_X\circ\text{id}_X\otimes u\circ\text{id}_X\otimes \eta_{g^{-1}}(\mathbf{1}))\mu(X,g^{-1}(\mathbf{1}))\circ\text{id}_{g(X)}\otimes (\nu^\mathbf{1}_{g,g^{-1}})^{-1}\\
&=g(\text{id}_X\otimes \eta_{g^{-1}}(\mathbf{1}))\mu(X,g^{-1}(\mathbf{1}))\circ\text{id}_{g(X)}\otimes (\nu^\mathbf{1}_{g,g^{-1}})^{-1}\\
&=\mu_g(X, \mathbf{1})\circ \text{id}_{g(X)}\otimes g(\eta_{g^{-1}}(\mathbf{1}))\circ  \text{id}_{g(X)}\otimes (\nu^\mathbf{1}_{g,g^{-1}})^{-1}\\
&=\mu_g(X, \mathbf{1})\circ\text{id}_{g(X)}\otimes\eta_g^{-1}=\text{id}_{g(X)}.
\end{align*}
Here, the second equality is by naturality of $\nu_{g,g^{-1}}$, the third and the sixth equalities are by naturality of $\mu_g$,   the forth equality is by naturality of $\eta_{g^{-1}}$, and the seventh equality follows as the  simple object $\mathbf{1}$ is a $G$-equivariant object. The second equation can be proved similarly.
\end{proof}

\begin{theo}\label{charactunit}Let $\C$ be a non-degenerate fusion category over a Tannakian fusion category $\E=\Rep(G)$. Then the following are equivalent:
\begin{enumerate}
\item [(1).]$[\C]=[\E]$.
 \item [(2).] $\C\cong\cZ(\A,\E)$, where $\A$ is a fusion category over $\E$.
 \item  [(3).] $\C$ contains a connected \'{e}tale algebra $A$ such that $\FPdim(\C)=\FPdim(\E)\FPdim(A)^2$ and $A\cap\E=\mathbf{1}$.
 \item [(4).] $\C_G\cong\cZ(\B)$ contains a $G$-equivariant Lagrangian algebra $A$.
\end{enumerate}
\end{theo}
\begin{proof}The first two conditions are equivalent   by the definition of Witt equivalence. We only need to prove they are equivalent to the third and fourth conditions.

$(3)\Rightarrow(1)$. Assume that $\C$ contains such a connected \'{e}tale algebra $A$. Then \cite[Corollary 4.6]{DNO} says that $\C\boxtimes_\E(\C_A^0)^\text{rev}\cong\cZ(\C_A,\E)$ as braided fusion category, notice that $\C_A^0$ is also a non-degenerate fusion category over $\E$ and $\FPdim(\C_A^0)=\frac{\FPdim(\C)}{\FPdim(A)^2}=\FPdim(\E)$. Therefore,
\begin{align*}
(\C_A^0)^\text{rev}\cong\C_A^0\cong\E
 \end{align*}
as symmetric fusion category. Hence, we have the following braided tensor   equivalences of fusion categories
\begin{align*}
\cZ(\C_A,\E)\cong\C\boxtimes_\E(\C_A^0)^\text{rev}\cong\C\boxtimes_\E\E^\text{rev}\cong\C,
\end{align*}
 which then implies  $[\C]=[\E]$ by definition.

$(1)\Rightarrow(3)$. If $\C\cong\cZ(\A,\E)$, where $\A$ is a fusion category over $\E$; in particular, $\frac{\FPdim(\A)^2}{\FPdim(\E)}=\FPdim(\C)$. Let $F_1$  be the restriction of forgetful functor $F:\cZ(\A)\to\A$ on $\cZ(\A,\E)$ and let $\I_1$ be its right adjoint functor. Since $\E$ is mapped fully faithful into $\A$ via forgetful functor $F$, $F_1$ is surjective by \cite[Theorem 3.12]{DNO}. Notice that $F_1$  obviously admits a central tensor functor structure, so $A:=\I_1(\mathbf{1})$ is a connected \'{e}tale algebra such that $\C_A\cong\A$ by \cite[Lemma 3.5]{DMNO}. Moreover, $\E\cap A=\mathbf{1}$, since for any simple object $X\in\E$, we have
\begin{align*}
\text{Hom}_\C(X,\I_1(\mathbf{1}))\cong \text{Hom}_\A(F_1(X),\mathbf{1})\cong \text{Hom}_\A(F_1(X),F_1(\mathbf{1}))\cong \text{Hom}_\E(X,\mathbf{1})\cong \delta_{X,\mathbf{1}}\K.
\end{align*}
And \cite[Lemma 3.11]{DMNO} says that $\FPdim(A)=\frac{\FPdim(\C)}{\FPdim(\A)}$, then
\begin{align*}
\FPdim(\E)\FPdim(A)^2= \FPdim(\E) \frac{\FPdim(\C)^2}{\FPdim(\A)^2}=\FPdim(\C).
\end{align*}
 Hence,  algebra $A$ is the desired connected \'{e}tale algebra.

$(4)\Rightarrow(1)$. If $\C_G\cong\cZ(\B)$ contains a $G$-equivariant Lagrangian algebra $A$. Then fusion category $\cZ(\B)_A$ admits a $G$-equivariant structure by Lemma \ref{G-equivalg}.
Let $\A:=(\cZ(\B)_A)^G$, so fusion category $\A$ is a fusion category over $\E$. Therefore, we deduce from \cite[Proposition 2.10]{ENO2} and \cite[Corollary 3.20]{DMNO} that there exist the following equivalences of braided fusion categories
\begin{align*}
\cZ(\A)_G^0\cong\cZ(\cZ(\B)_A)\cong\cZ(\B)\boxtimes(\cZ(\B)^0_A)^\text{rev}\cong\cZ(\B).
\end{align*}
The last equivalence is because that $A$ is a Lagrangian algebra of $\cZ(\B)$. Hence, the Drinfeld center $\cZ(\A)$ is a minimal extension of $\C\cong\E'_{\cZ(\A)}$ and then $[\C]=[\E]$ as desired.

$(2)\Rightarrow(4)$.   Since $\A$ is a fusion category over $\E=\Rep(G)$, $\B=\A_G$ is a fusion category.  Then we have a braided tensor equivalence  $\cZ(\A)_G^0\cong\cZ(\B)$. Moreover, we also have tensor equivalences $(\C_G)_B\cong\B\cong(\C_A)_G$, where $B$ is the Lagrangian algebra corresponding to forgetful functor $\cZ(\B)\to\B$. Notice that $\E$ is the M\"{u}ger center of $\C$, so $\E$ centralizes $A$. We claim that $B$ a $G$-equivariant algebra. Indeed,
\begin{align*}
\text{Hom}_\A(\mathbf{1},F_1(A))=\text{Hom}_\C(\I_1(\mathbf{1}), A)=\text{Hom}_\C(A, A),
\end{align*}
so $F_1(A)$ maybe not a connected algebra, where $F_1:\cZ(\A,\E)\to\A$ is the restriction of forgetful functor $F:\cZ(\A)\to\A$. However, let $F_2:\C\to\cZ(\B)=\C_G$ be the natural surjective braided tensor functor, since $\E\cap A=\mathbf{1}$, $A$ is mapped faithfully into $\C_G$ via tensor functor $F_2$. Meanwhile, it is easy to see that
\begin{align*}
\text{Hom}_\A(\mathbf{1},F_2(A))=\text{Hom}_\C(\I_2(\mathbf{1}), A)=\text{Hom}_\C(\text{Fun}(G), A)=\K,
 \end{align*}
where $\I_2$ is the  right adjoint functor of $F_2$. So $ A\cong F_2(A)\cong B$ is a connected Lagrangian algebra in $\cZ(\B)$.

Notice that $A\in \C\cong\cZ(\B)^G$ is a connected \'{e}tale algebra,  so $A$ admits a $G$-equivariant structure $(A,\zeta)$, where $\zeta={\{\zeta_g|\zeta_g:g(A)\overset{\sim}{\to} A}\}_{{g\in G}}$ is a natural isomorphism, here we regard $A$ as an object in $\cZ(\B)$. Let $\tau$ be the composition of isomorphisms $B\overset{\sim}{\to} F_2(A)\overset{\sim}{\to} A$. Let $\xi_g:=\tau^{-1}\circ \zeta_g\circ g(\tau)$, then it is easy to see that $\xi={\{\xi_g|\xi_g: g(B)\overset{\sim}{\to} B}\}_{g\in G}$ defines a $G$-equivariant structure on $B$, this completes the proof of the theorem.
\end{proof}
\begin{remk}Indeed, it is easy to see that items (1)--(3) of Theorem \ref{charactunit} are true for all symmetric fusion categories $\E$, moreover,   $\cZ(\A)_A^0\cong\cZ(\E)$ as braided fusion category, where $A$ is the Lagrangian algebra in Theorem \ref{charactunit}.
 In particular,   Theorem \ref{charactunit}    generalizes the conclusion of
 \cite[Proposition 5.8]{DMNO}, which characterizes when a non-degenerate fusion category $\C$ is braided tensor equivalent to   Drinfeld center of a fusion category.
\end{remk}

\section{$G$-quasi-monoidal categories}\label{Gqmc}

\subsection{Definitions}

Let $\C$ be a semisimple $\K$-linear category equipped with $\K$-bilinear binary tensor product $\otimes$; we assume the data of associativity isomorphisms $(-\otimes-)\otimes - \cong - \otimes (-\otimes -)$, but we allow the pentagon axiom to be violated.  We do assume the existence of the unit object defined as in \cite[2.1]{EGNO} --- this is an object ${\bf 1}\in \C$ equipped with isomorphism $\iota: {\bf 1}\otimes {\bf 1}\simeq {\bf 1}$ and such that the functors $X\mapsto X\otimes {\bf 1}$ and $X\mapsto {\bf 1}\otimes X$ are equivalences --- and for convenience we will assume that ${\bf 1}$ is simple.

Assume furthermore that the category
$\C$ is faithfully graded by a group $G$, $\C=\oplus_{g\in G}\C_g$ compatibly with the tensor product. Thus for any simple object $X\in \C$ we have
a well defined element $\deg(X)\in G$ and for any simple direct summand $Z$ of $X\otimes Y$ we have
$\deg(Z)=\deg(X) \deg(Y)$. The faithfulness means that for any $g\in G$ there exists at least one simple
object $X$ with $\deg(X)=g$.

\begin{defi}\label{quasim}
The category $\C$ as above is called $G$-quasi-monoidal if for all simple objects $X,Y,Z,T\in \C$ the pentagon diagram commutes up to a scalar and this scalar depends only on $\deg(X), \deg(Y), \deg(Z), \deg(T)\in G$.
\end{defi}

\begin{remk} The terminology is motivated by notion of $(G,\phi)$-quasiassociative algebra, see \cite{ABS}.
\end{remk}

\begin{ques} Assume that a semisimple category $\C$ is equipped with tensor product and associativity isomorphisms such that the pentagon diagram commutes up to a scalar for any quadruple of simple objects.
Is such category $G$-quasi-monoidal for some group $G$?
\end{ques}

There is an obvious notion of equivalence of $G$-quasi-monoidal categories: two such categories $\C$ and $\C'$ are equivalent if there is an equivalence $F: \C \simeq \C'$ preserving the grading and equipped with quasi-tensor structure (i.e.\ isomorphisms $F(X\otimes Y)\simeq F(X)\otimes F(Y)$) intertwining the associativity isomorphisms in both categories.

Let $\nu =\nu(\C): G^4\to \K^\times$ be the scalar which appears in the definition \ref{quasim} (so this is a function of 4 variables taking values in $G$). Note that $\nu(\C)$ is uniquely determined by the choice of associativity isomorphisms in $\C$ since we assumed that the $G$-grading is faithful. By definition $\nu(\C)$ is a 4-cochain on group $G$ with values in $\K^\times$.
Considering the associahedron diagram (see e.g. \cite[Appendix B]{CL}) we get

\begin{prop}
The cochain $\nu(\C)$ is 4-cocycle.
\end{prop}

{\bf Sketch of proof:} the vertices of the associahedron $K_5$ can be labeled by 5-fold tensor products
and the edges are given by various associativity maps, e.g. $$X\otimes ((Y\otimes Z)\otimes (T\otimes U))\to X\otimes (Y\otimes (Z\otimes (T\otimes U))).$$
We consider such diagram where $X,Y,Z,T,U$ are simple objects of a $G$-quasi-monoidal category.
The faces are either quadrangles which commute due to naturality of the associativity morphisms or pentagons
which commute up to a scalar by Definition \ref{quasim}. Thus we get a relation between scalars
appearing in Definition \ref{quasim}; one verifies that this relation says that the coboundary of  $\nu(\C)$
is trivial. $\square$

Thus we will call $\nu(\C)$ the {\em pentagonicity defect cocycle} of $\C$.

\begin{exam}
Assume $G$ is trivial group and let $\C$ be a $G$-quasi-monoidal category. Then each pentagon diagram with simple $X,Y,Z,T$ commutes up to a scalar $\nu$
which is the same for all $X,Y,Z,T$. Thus we can easily make $\C$ into monoidal category by defining a new associativity constraint to be the old one multiplied by $\nu^{-1}$.
\end{exam}

Similarly we can twist the associativity isomorphisms in $\C$ by an arbitrary 3-cochain $\lambda: G^3\to \K^\times$. Thus the $G$-quasi-monoidal category $\C^{\lambda}$ differs from $\C$ only by the associativity isomorphisms; for simple $X,Y,Z\in \C$ the associativity isomorphism $(X\otimes Y)\otimes Z\to X\otimes (Y\otimes Z)$
in $\C^{\lambda}$ equals the associativity isomorphism $(X\otimes Y)\otimes Z\to X\otimes (Y\otimes Z)$
in $\C$ multiplied by $\lambda(\deg(X), \deg(Y), \deg(Z))$.

\begin{prop}\label{diff}
The pentagonicity defect cocycle $\nu(\C^{\lambda})$ equals $\nu(\C)\partial \lambda$ where $\partial \lambda$
is the differential of $\lambda$.
\end{prop}

\begin{proof} The argument is identical to \cite[(60)]{ENO2}.
\end{proof}

We say that a simple object $X$ of $G$-quasi-monoidal category $\C$ is (left) rigid if there is a simple object $Y\in \C$ such that there exist morphisms ${\bf 1}\to X\otimes Y$ and $Y\otimes X \to {\bf 1}$ such that the compositions
$X\to (X\otimes Y)\otimes X \to X$ and $Y\to Y \otimes (X\otimes Y)\to Y$ are non-zero. We leave it to the reader to define the right rigid objects in $\C$.

\begin{remk} \label{cycl red}
Consider the cyclic subgroup $H\subset G$ generated by $\deg(X)$ and $H$-quasi-monoidal subcategory $\C(H)=\oplus_{g\in H}\C_g$. Since for any cyclic group $H$ we
have $\H^4(H,\K^\times)=0$, we
can find twisted category $\C(H)^{\lambda}$ such that its pentagonicity defect cocycle is trivial, i.e.\ $\C(H)^{\lambda}$ is monoidal category. It is clear that an object $X\in \C$ is rigid if and only if it is rigid in
$\C(H)^{\lambda}$. This shows that the properties of rigid objects in $G$-quasi-monoidal categories are similar to
properties of rigid objects in monoidal categories.
\end{remk}

We will say that a $G$-quasi-monoidal category is rigid if any simple object of $\C$ is both left and right rigid.
Similarly, we say that a $G$-quasi-monoidal category is weakly rigid if for any simple object $X$ there is a unique
simple object $Y$ such that $\mbox{Hom}({\bf 1}, X\otimes Y)$ is not zero and in this case
$\dim  \mbox{Hom}({\bf 1}, X\otimes Y)=1$ (equivalently, the Grothendieck ring of the category $\C$ is
a based ring in the sense of \cite[Definition 3.1.3]{EGNO}).

It is clear that any rigid $G$-quasi-monoidal category is weakly rigid. Conversely,
using Remark \ref{cycl red} and the result of T.~Deshpande and S.~Mukhopadhyay \cite[Corollary 2.11]{DeMu} we get the following result:

\begin{prop}\label{rigid}
A $G$-quasi-monoidal category $\C$ is rigid if an only if it is weakly rigid and the identity component $\C_e\subset \C$ is rigid. $\square$
\end{prop}

\subsection{Graded extensions}\label{gradext}
In this section we restate some results of \cite{ENO2} in the language of $G$-quasi-monoidal categories.

Let $\C$ be a $G$-quasi-monoidal category. Assume that the pentagonicty defect cocycle $\nu(\C)$ is normalized
(since any cocycle is cohomologeous to a normalized cocycle, this always can be achieved if we replace $\C$ by a suitable twisted category $\C^\lambda$).
Then the identity component $\C_e$ is a monoidal category and every summand $\C_g$ is $\C_e$-bimodule category.

\begin{prop}
Assume in addition that $\C$ is rigid. Then
\begin{enumerate}
\item[(i)] For any $g\in G$, the $\C_e$-bimodule category $\C_g$ is invertible.
\item[(ii)] For any $g,h\in G$, the tensor product in $\C$ induces an equivalence of $\C_e$-bimodule categories
$\C_g\boxtimes_{\C_e}\C_h\simeq \C_{gh}$.
\end{enumerate}
\end{prop}

\begin{proof} The proof is identical to proof of \cite[Theorem 6.1]{ENO2}.
\end{proof}

Thus any $G$-quasi-monoidal category determines a group homomorphism $G\to \BrPic(\C_e)$.
Moreover the arguments in \cite[8.3]{ENO2} show that this homomorphism upgrades to a homomorphism of
2-groups $G\to \underline{\BrPic}(\C_e)$ (equivalently, we have an isomorphism of $\C_e$-bimodule functors as in \cite[(49)]{ENO2}).

Conversely, given a homomorphism of 2-groups $G\to \underline{\mbox{BrPic}}(\C_e)$ we can construct a graded extension
with quasi-tensor product as in \cite[Theorem 8.4]{ENO2}. Moreover, for this quasi-tensor product we can choose
some associativity isomorphisms, see \cite[(54)]{ENO2}. Thus we obtain a $G$-quasi-monoidal category.
It follows from definitions that this category is weakly rigid; thus it is rigid by Proposition \ref{rigid}.
Thus for any fusion category $\D$ the constructions above give rise to a bijection
\begin{multline*}
\{ \text{2-group homomorphisms $G\to \underline{\mbox{BrPic}}(\D)$} \} \\ \leftrightarrow \\ \{ \text{rigid $G$-quasi-monoidal categories with $\C_e\simeq \D $ up to twisting} \}
\end{multline*}

In the context of \cite{ENO2} the cohomology class of the pentagonicity defect cocycle of $G$-quasi-monoidal category associated
with homomorphism $G\to \underline{\mbox{BrPic}}(\D)$ is called the associativity constraint obstruction class
$O_4$, see \cite[Definition 8.7]{ENO2}.

\section{Isomorphism $\eta_\E: \operatorname{Ker}(\phi_\E)\simeq \H^4(G,\K^\times)$}\label{isom eta}

Let $\E$ be a Tannakian fusion category. Choose an equivalence of symmetric fusion categories $\E \simeq \Rep(G)$
where $G$ is a finite group.
Let $M(\E)$ be the set (of equivalence classes) of non-degenerate braided fusion categories over $\E$.
For any $\C \in M(\E)$ we can consider its de-equivariantization $\C_G$, see \cite[subsection 4.2]{DrGNO2}. Thus $\C_G$ is a non-degenerate braided fusion category equipped with an action of $G$. We can reconstruct $\C$ from $\C_G$ using the procedure of equivariantization, see \cite[subsection 4.2]{DrGNO2}. Namely $\C =(\C_G)^G$ (equivalence of braided fusion categories over $\E$).

\begin{remk}
 Let $\tilde M_G$ be the set (of equivalence classes) of non-degenerate braided fusion categories equipped with an action of $G$. The de-equivariantization and equivariantization constructions  give a bijection
$M(\E)=\tilde M_G$, see \cite[Theorem 4.18]{DrGNO2}.
\end{remk}

Consider a subset $M(\E)^\text{tr}\subset M(\E)$ consisting of categories $\C$ such that the Witt class of $\C_G$ is trivial. Equivalently, $\C \in M(\E)^\text{tr}$ if and only if there exists a fusion category $\A$ and an equivalence of braided fusion categories $\C_G\simeq \cZ(\A)$. Note that $\A$ is not necessarily unique. Let us transport the action of $G$ on $\C_G$ to action of $G$ on $\cZ(\A)$ using this equivalence. Thus we have a homomorphism of 2-groups
$G\to  \underline{\mbox{Aut}}(\cZ(\A))$ where $\underline{\mbox{Aut}}(\cZ(\A))$ is 2-group of braided autoequivalences of the category $\cZ(\A)$. Recall that by \cite[Theorem 1.1]{ENO2} there is a canonical isomorphism of
2-groups $\underline{\mbox{Aut}}(\cZ(\A))\simeq  \underline{\mbox{BrPic}}(\A)$, so by taking a composition we get a homomorphism of 2-groups $G\to  \underline{\mbox{BrPic}}(\A)$. Thus by using the results of Section \ref{gradext}
we can construct $G$-quasi-monoidal category $\tilde \A$ with $\tilde \A_e\simeq \A$. Let $\nu$ be the pentagonicity defect cocycle of the category $\tilde \A$. The cocycle $\nu$ depends on the following choices:
we choose fusion category $\A$, braided equivalence $\C_G \simeq \cZ(\A)$, and associativity isomorphisms
in the category $\tilde \A$. For any $\C \in M(\E)^\text{tr}$ let $\text{Ch}(\C)$ denote the set of such choices.
Now for $\C \in M(\E)^\text{tr}$ and $c\in \text{Ch}(\C)$ (note that the set $\text{Ch}(\C)$ is non-empty) we define
\begin{equation}
\tilde \eta_\E(\C, c)=\nu
\end{equation}

\begin{remk}\label{hat eta}
There is an alternative definition of the cocycle $\hat \eta (\C)$ as an obstruction to construct a
$G$-crossed category starting from category $\C_G$ (with its canonical $G$-action) as its trivial component,
see \cite[Theorem 7.12]{ENO2} (note that this definition makes sense for any $\C \in M(\E)$, not just
$\C \in M(\E)^\text{tr}$). An interested reader is invited to state and prove the counterparts of the results
of this Section using $\hat \eta$ in the place of $\tilde \eta$.
\end{remk}

\begin{lemm}\label{etaop}
\begin{enumerate}
\item[(i)] For any $\C \in M(\E)^\text{tr}$ and $c\in \text{Ch}(\C)$ there is $c'\in \text{Ch}(\C^\text{rev})$ such that $\tilde \eta_\E(\C^\text{rev}, c')=\tilde \eta_\E(\C, c)^{-1}$;

\item[(ii)]  For any $\C, \D \in M(\E)^\text{tr}$ and $c\in \text{Ch}(\C), d\in \text{Ch}(\D)$ there is $e\in \text{Ch}(\C \boxtimes_\E \D)$ such that
$\tilde \eta_\E(\C \boxtimes_\E \D, e)=\tilde \eta_\E(\C, c)\tilde \eta_\E(\D, d)$.
\end{enumerate}
\end{lemm}

\begin{proof} (i) Consider $\tilde \A^{\mathrm{op}}$, i.e.\ the opposite category of $\tilde \A$

(ii) Assume $\C$ and $c$ lead to the $G$-quasi-monoidal category~$\tilde \A$ and $\D$ and $d$ lead to the category~$\tilde \B$. The external tensor product $\tilde \A \boxtimes \tilde \B$ (see e.g. \cite[1.11]{EGNO})
has an obvious grading by the group $G\times G$ and the associativity isomorphism making it into rigid
$G\times G$-quasi-monoidal category; its pentagonicity defect cocycle is the (external) product of $\tilde \eta_\E(\C, c)$ and $\tilde \eta_\E(\D, d)$.

Now consider the subcategory $\oplus_{g\in G} \tilde \A_g\boxtimes \tilde \B_g\subset \tilde \A \boxtimes \B$;
 it is a rigid $G$-quasi-monoidal category with  pentagonicity defect cocycle $\tilde \eta_\E(\C, c)\tilde \eta_\E(\D, d)$. By the bijection from Section \ref{gradext} this category will come from some choice
$e\in \text{Ch}(\C \boxtimes_\E \D)$.
\end{proof}

Now let  $M(\E)^{G\text{-tr}}\subset M(\E)$ be the subset consisting of categories $\C$ with trivial class in $\W(\E)$,
that is $[\C]=[\E]$. Then $[\C_G]=\phi_\E ([C])=0$, so $M(\E)^{G\text{-tr}}\subset M(\E)^\text{tr}$.

\begin{lemm}\label{etatriv}
For any $\C \in M(\E)^{G\text{-tr}}$ and $c\in \text{Ch}(\C)$ we have $\tilde \eta_\E(\C,c)$ is a coboundary.
\end{lemm}

\begin{proof} The assumption $[\C]=[\E]$ means that $\C\cong\Y(\A,\E)$, where $\A$ is a fusion category over $\E$, see Theorem \ref{charactunit}. By definition this says that there is an embedding of braided fusion categories
$\C\subset \cZ(\A)$ such that $\C'_{\cZ(\A)}=\E \subset \C$, or, equivalently, $\E'_{\cZ(\A)}=\C$. Thus
the de-equivariantization of $\E'_{\cZ(\A)}$ by $\E$ is equivalent to $\C_G$. Now any choice
$c\in \text{Ch}(\C)$ includes a choice of an equivalence $\C_G \simeq \cZ (\B)$ for some fusion category $\B$.
Thus by Theorem \cite[Theorem 1.3]{ENOw} there is a $G$-extension $\tilde \B$ and a braided equivalence
$\cZ(\A)\cong \cZ(\tilde \B)$. We have two $G$-actions on the category $\cZ(\B)\simeq \C_G$: the canonical
action of $G$ on $\C_G$ and the action $G\to \underline{\mbox{BrPic}}(\B)=\underline{\mbox{Aut}}(\cZ(\B))$
coming from the existence of $G$-extension $\tilde \B$. However it follows from \cite[Theorem 3.5]{GNN}
that these two actions coincide. Thus at least one choice $c'\in \text{Ch}(\C)$ starting with the equivalence
$\C_G \simeq \cZ (\B)$ leads to the category $\tilde \B$ considered as $G$-quasi-monoidal category,
so $\eta_\E(\C,c')$ is trivial. Any other choice $c\in \text{Ch}(\C)$ starting with the same equivalence
$\C_G \simeq \cZ (\B)$ leads to a category which differs from $\tilde \B$ by twisting (since the only thing
to choose is the associativity isomorphisms, see \cite[8.6]{ENO2}. Thus Lemma follows by Proposition \ref{diff}.
\end{proof}

For a pair $(\C,c)$ where $\C \in M(\E)^\text{tr}$ and $c\in\text{Ch}(\C)$ consider the cohomology class $\bar \eta_\E(\C,c)\in \H^4(G, \K^\times)$ of the cocycle $\tilde \eta_\E(\C,c)$.

\begin{coro}
The class $\bar \eta_\E(\C,c)$ does not depend on $c$.
\end{coro}

\begin{proof} Let $c_1, c_2\in Ch(\C)$. By Lemma \ref{etaop} (i) there is $c'\in Ch(\C^\text{rev})$ such that $\tilde \eta_\E(\C^\text{rev}, c')=\tilde \eta_\E(\C, c_1)^{-1}$. Thus by Lemma \ref{etaop} (ii) there is $e\in Ch(\C \boxtimes_\E \C^\text{rev})$ such that $$\tilde \eta_\E(\C \boxtimes_\E \C^\text{rev}, e)=\tilde \eta_\E(\C, c_2)\tilde \eta_\E(\C^\text{rev}, c')=
\tilde \eta_\E(\C, c_2)\tilde \eta_\E(\C, c_1)^{-1}.$$ On the other hand the category $\C \boxtimes_\E \C^\text{rev}\in M(\E)^{G\text{-tr}}$ (by definition of the inverse element in $\W(\E)$). Thus by Lemma \ref{etatriv} we have
that $\tilde \eta_\E(\C \boxtimes_\E \C^\text{rev}, e)$ is a coboundary. Hence the cohomology classes of $\tilde \eta_\E(\C, c_1)$ and $\tilde \eta_\E(\C, c_2)$ coincide as claimed.
\end{proof}

Thus there is a well defined class $\bar \eta_\E(\C)\in \H^4(G, \K^\times)$ such that $\bar \eta_\E(\C, c)=\bar \eta_\E(\C)$ for any $c\in Ch(\C)$.

\begin{coro} The map $\bar \eta_\E : M(\E)^\text{tr}\to \H^4(G, \K^\times)$ satisfies the following:
\begin{enumerate}
\item[(i)] For any $\C, \D\in M(\E)^\text{tr}$ we have $\bar \eta_\E(\C \boxtimes_\E \D)=\bar \eta_E(\C)\bar \eta_\E(\D)$;
\item[(ii)] For any $\C \in M(\E)^{G\text{-tr}}$ we have $\bar \eta_\E(\C)=0$.
\end{enumerate}
\end{coro}

\begin{proof} (i) is immediate from Lemma \ref{etaop} (ii) and (ii) follows from Lemma \ref{etatriv}.
\end{proof}

It follows from the definition of $\W(\E)$ that there is a unique homomorphism $\eta_\E: \operatorname{Ker}(\phi_\E)\to \H^4(G, \K^\times)$ such that $\eta_\E([\C])=\bar \eta_\E(\C)$.

We have a converse to Lemma \ref{etatriv}:

\begin{lemm}\label{injmap}
Assume $\bar \eta_\E(\C)=0$. Then $\C \in M(\E)^{G\text{-tr}}$. Equivalently, the homomorphism $\eta_\E$ is injective.
\end{lemm}

\begin{proof}
Since $\bar \eta_\E(\C)=0$, for some choice $c\in \text{Ch}(\C)$ the resulting $G$-quasi-monoidal  category $\tilde\A$
has trivial pentagonicity defect cocycle, so it is a $G$-graded fusion category. Then it follows from \cite[Theorem 3.5]{GNN} that $\C$ is a braided fusion  subcategory of $\cZ(\tilde\A)$ with M\"{u}ger center being $\E$, so $\C\cong\cZ(\tilde\A,\E)$ as claimed.
\end{proof}

As explained in \cite{JR} the main idea of the proof of the next proposition is contained in \cite{WWW}; a version of similar argument involving $\H^3(G,\K^\times)$ is contained in \cite{EG}.

\begin{prop}\label{surjmap}
The map $\eta_E: \operatorname{Ker}(\phi_\E)\to \H^4(G, \K^\times)$ is surjective.
\end{prop}

\begin{proof} For a class $\omega \in \H^4(G, \K^\times)$ let $\nu \in Z^4(G, \K^\times)$ be a representing cocycle. Let $\tilde G$ be a finite group with a surjective homomorphism $p: \tilde G\to G$ such that $p^*\omega =0$, see \cite{opolka}. Thus we can pick a cochain $\psi \in C^3(\tilde G, \K^\times)$ such that $p^*\nu=\partial \psi$.

Consider the category $\vvec_{\tilde G}$ with the usual tensor product (so tensor product of simple objects labeled by $g, h\in \tilde G$ is labeled by $gh$). Let us use $\psi$ as an associativity isomorphism for this category and let
us use $p$ to define $G$-grading on this category (so a simple object labeled by $g\in \tilde G$ has degree
$p(g)\in G$). The resulting category is $G$-quasi-monoidal category with pentagonicity defect cocycle $\nu$.
Thus we proved that $\omega$ is in the image of $\eta_\E$.
\end{proof}

Combining Lemma \ref{injmap} and Proposition \ref{surjmap} we get

\begin{theo}\label{isomap}
The map $\eta_\E: \operatorname{Ker}(\phi_\E)\to \H^4(G, \K^\times)$ is an isomorphism. \qed
\end{theo}

Combining this with Corollary \ref{kernelminext} we get the following

\begin{coro}\label{h4minext}
The minimal extension conjecture holds for all braided fusion categories
over Tannakian category $\E=\Rep(G)$ if and only if $\H^4(G,\K^\times)=0$. \qed
\end{coro}

Theorem \ref{Wittgroup} and Theorem \ref{isomap} together imply
\begin{theo}\label{isomptheo}
Let $\E=\Rep(G)$ be Tannakian. Then $\W(\E)\cong\W\oplus \H^4(G,\K^\times)$. \qed
\end{theo}

\begin{remk} One verifies that the projection $\W(\E)\cong\W\oplus H^4(G,\K^\times)\to H^4(G,\K^\times)$
can be described as follows: the class $[\C]\in \W(\E)$ is sent to the cohomology class of the cocycle
$\hat \eta(\C)$ from Remark \ref{hat eta}.
\end{remk}

Let $\C$ be a braided fusion category over a symmetric fusion category $\E$. For any $n\geq1$, denote
\begin{align*}
\C^{\boxtimes_\E^ n}:=\overbrace{\C\boxtimes_\E \cdots\boxtimes_\E\C}^n,
 \end{align*}

 We have the following result  on the minimal extension of braided fusion categories.

 \begin{coro}\label{existMinExten}Let $\C$ be an arbitrary  non-degenerate fusion category over $\E=\Rep(G)$, then there exists a positive  integer $n$ such that   $\C^{\boxtimes_\E ^n}$ admits a minimal extension.
 In fact we can take $n=|G|$.
 \end{coro}

 \begin{proof}
  \cite[Theorem 3.2]{OY} states that a category $\C \in M(\E)$ admits a minimal extension if and only if
  $[\C]\in \W_E$ is in the image of homomorphism $S_\E$. Now
Theorem \ref{Wittgroup} implies that the image of $S_\E$ is precisely the kernel of projection
$\W(\E) \to \operatorname{Ker}(\phi_\E)\simeq \H^4(G,\K^*)$. It is well known that all elements
of $\H^4(G,\K^\times)$ are torsion of order dividing $|G|$. Then the result follows since
\begin{align*}
 [\C^{\boxtimes_\E ^n}]=[\C]^n.
 \end{align*}
 \end{proof}

\section{Group homomorphism $\varphi_\E:\W(\E)\to s\W$}\label{sectsuperTann}
In this section, let $\E=\text{Rep}(\widetilde{G},z)$ be a super-Tannakian fusion category, where $z\in \widetilde{G}$ is a non-trivial central element of order $2$, and $z$ acts as a parity automorphism on $\text{Rep}(\widetilde{G})$. Then $\E$ contains a maximal Tannakian fusion subcategory  $\Rep(G)$, where $G:=\widetilde{G}/\langle z\rangle$. Moreover, we have the following commutative diagram of Witt groups
\begin{align*}
\xymatrix{
  \W \ar[rr]^{S_\E} \ar[dr]_{S}
                &  &    \W(\E) \ar[dl]^{\varphi_\E}    \\
                & s\W                }\label{superdiag}
\end{align*}
where $S:=S_\text{sVec}:\W\to s\W$, $S([\C])=[\C\boxtimes \sVec]$ and $\varphi_\E([\D])=[\D_G]$ for any $[\C]\in\W$, $[\D]\in\W(\E)$. Thus  $\varphi_\E$ is surjective as $S$ is surjective by \cite{JR}. Notice that $\operatorname{Ker}(S_\E)$ is the subgroup generated by Witt equivalence classes of minimal extensions of $\E$ and $\operatorname{Ker}(S)=\W_\text{Ising}=\Z_{16}$ \cite{DNO}, where $\W_\text{Ising}\subseteq\W$ is the subgroup generated by Witt equivalence class $[\I]$ of an Ising category $\I$. Hence, $\operatorname{Ker}(S_\E)$ is a subgroup of $\Z_{16}$.
\begin{ques}
Can we give a complete characterization the order of $\operatorname{Ker}(S_\E)$ for an arbitrary super-Tannakian category $\E$?
\end{ques}
Note that we have an exact sequence of groups
\begin{align*}
1\to\operatorname{Ker}(\varphi_\E)\to\W(\E)\overset{\varphi_\E}{\rightarrow} s\W\to1.
\end{align*}
Hence the subgroup $\operatorname{Ker}(\varphi_\E)$ is crucial in determining the structure of $\W(\E)$. By definition,  $\varphi_\E([\C])=[\text{sVec}]$ if and only if $ \C_G\cong \cZ(\A, \text{sVec})$, where $\A$ is a fusion category over $\text{sVec}$ and $\cZ(\A, \text{sVec})$ is the centralizer of $\sVec$ in $\cZ(\A)$.

\begin{prop}If $\E\cong \text{sVec}\boxtimes \Rep(G)$,  then there  is a  group homomorphism $\psi_\E:s\W\to \W(\E)$ such that $\varphi_\E\circ\psi_\E=\text{id}$. Hence $\W(\E)\cong s\W\oplus \operatorname{Ker}(\varphi_\E)$.
\end{prop}
\begin{proof}Denote $\E_2:=\Rep(G)$. For any elements $[\C],[\D]\in s\W$, by definition, let $A_1, A_2$ be the connected \'{e}tale algebras such that $(\text{sVec}\boxtimes \text{sVec})_{A_1}\cong \text{sVec}$ and $(\E_2\boxtimes \E_2)_{A_2}\cong \E_2$, respectively. Then, $(\E\boxtimes\E)_A\cong\E$, where $A:=A_1\boxtimes A_2$ and $\C\boxtimes_\text{sVec}\D=(\C\boxtimes\D)_{A_1}$,
\begin{align*}
((\C\boxtimes_\text{sVec}\D)\boxtimes \E_2)\boxtimes_{\E_2}\E_2
&=(((\C\boxtimes_{\text{sVec}}\D)\boxtimes \E_2)\boxtimes \E_2)_{A_2},\\
(\C\boxtimes \E_2)\boxtimes_\E(\D\boxtimes \E_2)&=(\C\boxtimes \E_2\boxtimes\D\boxtimes \E_2)_A.
\end{align*}
We define $\psi_\E([\C])=[\C\boxtimes\E_2]$, $[\C]\in s\W$. Obviously, $\psi_\E$ is well-defined and
\begin{align*}
\varphi_\E(\psi_\E([\C]))=\varphi_\E([\C\boxtimes\E_2])=[(\C\boxtimes\E_2)_G]=[\C].
\end{align*} Note that \'{e}tale algebras $A_1,A_2$ are centralizing each other. Then for $[\C],[\D]\in s\W$,\begin{align*}
&\psi_\E([\C])\psi_\E([\D])\\&=[(\C\boxtimes \E_2)\boxtimes_\E(\D\boxtimes \E_2)]
\\&=[(\C\boxtimes \E_2\boxtimes\D\boxtimes \E_2)_A]
\\&=[((\C\boxtimes \E_2\boxtimes\D\boxtimes \E_2)_{A_2})_{A_1}]
\\&=[(\C\boxtimes\D\boxtimes \E_2)_{A_1}]
\\&=[ \C\boxtimes_{\text{sVec}}(\D\boxtimes \E_2)]
\\&=[(\C\boxtimes_{\text{sVec}}\D)\boxtimes \E_2]=\psi_\E([\C][\D]).
\end{align*}
this finishes the proof of the proposition.
\end{proof}

However, an arbitrary super-Tannakian fusion category $\E$ cannot be factored as a Deligne tensor product $\text{sVec}\boxtimes \Rep(G)$, and thus  the  morphism $\psi_\E: s\W\to\W(\E)$ is not well-defined. We end this section by proposing the following questions.

\begin{ques}
For a super-Tannakian fusion category $\E\ncong\sVec$, is $\operatorname{Ker}(\varphi_\E)$  finite? Noting that $s\W$ is $2$-primary and the maximal finite order of an element in $s\W$ is $4$ by \cite{DNO}, and supposing that the order of $[\C]\in\W(\E)$ is finite, then $\varphi_\E([\C^{\boxtimes_\E^4}])=\varphi_\E([\C])^4=[\sVec]$, so $[\C^{\boxtimes_\E^4}]\in\operatorname{Ker}(\varphi_\E)$. Is   the torsion part of $\operatorname{Ker}(\varphi_\E)$ (or $\W(\E)$) also 2-primary?
\end{ques}

\author{Theo Johnson-Freyd\\ \thanks{Email:\,theojf@dal.ca}
\\{\small Department of Mathematics \& Statistics
Dalhousie University,
Halifax, NS, Canada}
\\{\small Perimeter Institute for Theoretical Physics, Waterloo, ON, Canada}}
\\
\author{Victor Ostrik \\ \thanks{Email:\,vostrik@uoregon.edu}
\\{\small Department  of Mathematics, University of Oregon, Eugene, OR, USA}}\\
\author{Zhiqiang Yu\\ \thanks{Email:\,zhiqyumath@yzu.edu.cn}\\{\small School  of Mathematics, Yangzhou University, Yangzhou 225002, China}}

\end{document}